\documentclass[a4 paper,10pt,twoside]{article}

\usepackage[centertags]{amsmath}
\usepackage{amssymb}
\usepackage{amsfonts}
\usepackage{amsthm}
\usepackage{newlfont}
\usepackage{amscd}
\usepackage[numbers]{natbib}
\usepackage{tikz}
\usepackage[top=1.7in, bottom=1.7in, left=1.5in, right=1.5in]{geometry}
\title{\large Euler’s  $\ell$-totients and Riemann hypothesis}
\author{}
\date{}
\theoremstyle{plain}
\newtheorem{theorem}{Theorem}[section]

\newtheorem{lemma}[theorem]{Lemma}
\newtheorem{corollary}[theorem]{Corollary}

\theoremstyle{example}

\numberwithin{equation}{section}
\theoremstyle{definition}
\newtheorem{definition}[theorem]{Definition}
\theoremstyle{remark}
\newtheorem{remark}[theorem]{Remark}
\begin{document}
\maketitle
\begin{center}
\textbf{Ahmed Gaber  }\vspace{0.15cm}\\
 Mathematics Education Program, Faculty of Education and Arts, Sohar 3111, Oman\\
 Department of Mathematics, Faculty of Science, Ain Shams University, Cairo, Egypt\\
  a.gaber@sci.asu.edu.eg\\
\end{center}
\vspace{0.1 cm}
\begin{abstract}
        This papers developes a  new analytic lens for investigating the Riemann hypothesis. For each fixed integer $\ell \ge1$ define the Euler’s  $\ell$-totient function $\varphi_\ell$ by
\[
\varphi_\ell(n):=n\prod_{\substack{p\ \text{prime}\\ v_p(n)\ge \ell}}\Big(1-\frac{1}{p}\Big)
\]
and its summatory function $\Phi_\ell(x):=\sum_{n\le x}\varphi_\ell (n)$. An analytic treatment of the generalized Euler totient function $\phi_\ell(n)$, including its Euler product expansion, meromorphic continuation and pole structure is developed. For each $\ell$, a necessary criterion and a sufficient criterion for the Riemann hypothesis are obtained via $\Phi_\ell(x)$.
\end{abstract}
\begin{flushleft}
\end{flushleft}
\begin{flushleft}
\end{flushleft}

\section{Introduction and motivation}
The Riemann hypothesis (RH) has resisted proof for over 150 years. Having multiple equivalent formulations reveals the hypothesis’s deep
connections across mathematics and provides more avenues for potential proof. Different formulations may be more amenable to different techniques, and over the years a vast landscape of equivalences has been discovered, spanning analysis, number theory, and even probability and group theory. Notable examples include the Nyman–Beurling criterion and its various generalizations \cite{BBLS2002, DarsesHillion2021, DeRoton2007, Yang2016, Yang2025}, recent criteria connected to the generalized Riemann hypothesis \cite{GargMaji2023}, formulations involving the Lindelöf hypothesis for primes and general sequences \cite{banks2022riemann, broucke2024lindelof,  gonek2020lindelof}, as well as equivalences stated in terms of highly composite numbers, integral equalities, permutation groups, and abundant numbers \cite{cavendish2024, deleglise2019, nazardonyavi2014, nicolas2022, sekatskii2023some}. The two volumes \cite{Broughan2017I, Broughan2017II} compile many of these diverse equivalences, illustrating the remarkable richness of the problem. In this paper, we introduce a new family of equivalent criteria parameterized by a positive integer \(\ell \geq 1\), arising from the study of  \(\phi_\ell\) and its summatory function \(\Phi_\ell\).

The introduction of the function \(\phi_\ell\) is motivated by the following idea: consider the set $B_\ell(n)$ consists of all integers from $1$ to $n$ that are not divisible by any prime appearing with multiplicity at least $\ell$ in the factorization of $n$. That is,   
\[
B_\ell(n) := \{\, m \in \{1,2,\dots,n\} \mid \text{for every prime } p \text{ with } p^\ell \mid n,\ p \nmid m \,\}.
\]
The set $B_\ell(n)$ provides a combinatorial interpretation of the Euler’s  $\ell$-totient (simply,   $\ell$-totient) function $\phi_\ell(n)$. Indeed, as proved in Theorem \ref{t1},
\[
\phi_\ell(n) = |B_\ell(n)|.
\]
This interpretation generalizes the classical fact that $\phi(n)$ counts integers coprime to $n$. 
The parameter $\ell$ acts as a threshold: only primes with exponent at least $\ell$ in $n$ impose a restriction. 
Thus $\phi_\ell(n)$ measures how many integers up to $n$ avoid the ``$\ell$‑powerful part'' of $n$.
The $\ell$-totient interpolates between: The Mertens function (via its Dirichlet series structure), average behavior of integers with restricted prime power divisibility, and explicit formulas involving zeta zeros.

The study of  $\phi_\ell$ is also motivated by seeking profound connections to RH that reveal new structural insights into number theory. The parameter $\ell$ modulates how the zeros of the Riemann zeta function $\zeta(s)$ affect the error term. Studying the dependence on $\ell$ might reveal subtle properties of zeta zeros that are not apparent from single formulations. The family provides a hierarchy of statements to test numerically. If RH were false, it might manifest more clearly in some $\ell$-values than others.

Let $M(x) := \sum_{n \le x} \mu(n).$ The 
littlewood classical equivalence $M(x) = O_{}(x^{1/2+\varepsilon})$ if and only if RH is true connects the Möbius function directly to zeta zeros, see see \citep[p.~370]{Titchmarsh1986}. Our result utilizes  this in a non-trivial way: instead of studying $\mu(n)$ alone, we study the convolution $\phi_\ell(n) = n \sum_{d^\ell \mid n} \frac{\mu(d)}{d}$, which weights the Möbius function by arithmetic structure.
The asymptotic analysis of $\Phi_\ell(x)$ uncovers a beautiful structure: the main term comes from the pole at $s=2$, while the error term comes from poles at $s = 1 + \frac{\rho-1}{\ell}$, where $\rho$ are zeta zeros. Under RH, these poles align on the line $\Re(s) = 1 - \frac{1}{2\ell}$. This reveals that different $\ell$ values rescale the critical line in the complex plane. Thus the $\ell$-totient provides a microscope zooming in on different regions of the zeta function.

 The main result of this paper is Theorem \ref{main}, which establishes a new and flexible family of criteria for the Riemann hypothesis (RH) through the asymptotic behavior of the summatory function \(S_{\ell}(x)\). Specifically, we prove two complementary statements. First, assuming RH, we obtain for any fixed \(\ell \geq 1\) the estimate \(S_{\ell}(x) = \frac{x^2}{2\zeta(\ell+1)} + O_{\ell,\epsilon}(x)\); i.e., the error term is at most linear. Second, and more significantly, we show that for any integer \(\ell \geq 1\), if the error term satisfies the stronger bound 
\[
S_{\ell}(x) = \frac{x^2}{2\zeta(\ell+1)} + O\!\left(x^{1-\frac{1}{2\ell}+\epsilon}\right) \qquad (\text{for every } \epsilon>0),
\]
then the Riemann hypothesis must hold. This condition is not a direct converse of the first part but rather a sharper asymptotic requirement that forces all non‑trivial zeros of the Riemann zeta function to lie on the critical line. The exponent \(1-\frac{1}{2\ell}\) arises directly from the relationship between the parameter \(\ell\) and the nontrivial zeros \(\rho\) of the zeta function, as the poles of the associated Dirichlet series occur at points whose real parts are \(1 + \frac{\Re(\rho)-1}{\ell}\). This shows how \(\ell\) acts as a scaling parameter, transforming the critical line into a family of vertical lines and thereby offering multiple independent analytical perspectives on the Riemann hypothesis.
While the classical case $\ell=1$ recovers a known connection between the Mertens function and RH, the cases $\ell>1$ provide genuinely new analytic criteria.

\section{Preliminary results}
\begin{theorem}\label{t1}
For every integer $\ell \geq 1$ and every $n \in \mathbb{N}$,
\[
\phi_\ell(n) = |B_\ell(n)|.
\]
\end{theorem}
\begin{proof}
Let
\[
P_\ell(n) := \{ p \in \mathbb{P} \mid p^\ell \mid n \}.
\]
\textbf{Case 1: $P_\ell(n) = \emptyset$.} \\
Then by definition $\phi_\ell(n) = n$ (the empty product equals $1$). 
The condition defining $B_\ell(n)$ is vacuously true, so $B_\ell(n) = \{1,2,\dots,n\}$ and $|B_\ell(n)| = n$. 
Hence $\phi_\ell(n) = |B_\ell(n)|$.\\
\textbf{Case 2: $P_\ell(n) \neq \emptyset$.} \\
Write $P_\ell(n) = \{p_1,p_2,\dots,p_r\}$ with $p_1 < p_2 < \dots < p_r$.
Set
\[
Q := \prod_{i=1}^{r} p_i.
\]
Since each $p_i \in P_\ell(n)$ satisfies $p_i^k \mid n$, we have $p_i \mid n$, and because the primes are distinct, $Q \mid n$. 
Write $n = Q \cdot t$ for some positive integer $t$. Now, partition the set $\{1,2,\dots,n\}$ into $t$ consecutive blocks of length $Q$:
\[
\{1,\dots,Q\},\ \{Q+1,\dots,2Q\},\ \dots,\ \{(t-1)Q+1,\dots,tQ\}.
\]
Each block is a complete set of residues modulo $Q$. 
For an integer $m$, the condition ``$p_i \nmid m$ for all $i=1,\dots,r$'' is equivalent to $\gcd(m,Q)=1$, because the prime factors of $Q$ are exactly $p_1,\dots,p_r$.

In any complete set of residues modulo $Q$, the number of integers coprime to $Q$ is given by Euler's classical totient function
\[
\varphi(Q) = Q \prod_{p \mid Q} \left(1 - \frac{1}{p}\right).
\]
Since $Q$ is square‑free and its prime divisors are precisely $p_1,\dots,p_r$, we have
\[
\varphi(Q) = Q \prod_{i=1}^{r} \left(1 - \frac{1}{p_i}\right)
          = Q \prod_{p \in P_\ell(n)} \left(1 - \frac{1}{p}\right).
\]

Every block contains exactly $\varphi(Q)$ integers that are coprime to $Q$, i.e., satisfy the condition of $B_\ell(n)$. 
As there are $t$ disjoint blocks covering all numbers from $1$ to $n$, we obtain
\[
|B_\ell(n)| = t \cdot \varphi(Q)
         = \frac{n}{Q} \cdot Q \prod_{p \in P_\ell(n)} \left(1 - \frac{1}{p}\right)
         = n \prod_{p \in P_\ell(n)} \left(1 - \frac{1}{p}\right)
         = \phi_\ell(n).
\]
Thus, in both cases, $\phi_\ell(n) = |B_\ell(n)|$.
\end{proof}
\begin{definition}
For a fixed integer $\ell \geq 1$ and $n \in \mathbb{N}$, the Euler’s  $\ell$-totient $\phi_\ell(n)$ is defined by
\[
\phi_\ell(n) := n \prod_{\substack{p \in \mathbb{P} \\ p^{\ell} \mid n}} \left(1 - \frac{1}{p}\right),
\]
where $\mathbb{P}$ denotes the set of prime numbers.
\end{definition}

When $\ell=1$, $\mathcal{P}_1(n)=\{p: p\mid n\}$ and the identity reduces to the familiar Euler totient formula
\(
\varphi_1(n)=n\prod_{p\mid n}\big(1-\frac{1}{p}\big).
\)
For general $\ell$, only primes $p$ with exponent $\ge \ell$ in the factorization of $n$ contribute a $(1-1/p)$ factor.
\begin{lemma}
\begin{enumerate}
    \item The function $\phi_\ell$ is multiplicative.
    \item The value of  $\varphi_\ell$ at prime powers is calculated as follows:
\begin{equation}\label{123}
\varphi_\ell(p^a)=
\begin{cases} 
p^a, & 0\le a<\ell,\\[2mm]
p^a\Big(1-\dfrac{1}{p}\Big), & a\ge \ell.
\end{cases}
\end{equation}
\end{enumerate}
 
\end{lemma}
\begin{proof}
(1) Let  $\gcd(m,n)=1$. Since the condition $p^\ell \mid mn$ depends only on the prime factors of $m$ and $n$ separately, and because $\gcd(m,n)=1$, the sets $P_\ell(m)$ and $P_\ell(n)$ are disjoint. Then
\begin{eqnarray*}
    \phi_\ell(mn) &=& mn \prod_{p \in P_\ell(mn)} \left(1-\frac1p\right)\\
          &=& mn \left(\prod_{p \in P_\ell(m)} \left(1-\frac1p\right)\right)
                \left(\prod_{p \in P_\ell(n)} \left(1-\frac1p\right)\right)\\
          &=& \phi_\ell(m)\,\phi_\ell(n). 
\end{eqnarray*}
          
(2) For $n=p^a$ there are two cases. If $a<\ell$, then $\mathcal{P}_\ell(p^a)=\emptyset$, so the empty product equals $1$ and thus $\varphi_\ell(p^a)=p^a$. If $a\ge \ell$, then $\mathcal{P}_\ell(p^a)=\{p\}$. Hence, $\varphi_\ell(p^a)=p^a(1-1/p)$.          
\end{proof}

\section{Dirichlet series and meromorphic continuation of $\varphi_\ell$ }

\begin{theorem}[Euler product for $\phi_\ell$] \label{thm:euler_product}
For every integer $\ell \ge 1$ and for $\Re(s) > 2$, the Dirichlet series of $\phi_\ell(n)$ admits the Euler product
\[
\sum_{n=1}^{\infty} \frac{\phi_\ell(n)}{n^{s}} 
= \frac{\zeta(s-1)}{\zeta\!\bigl(\ell(s-1)+1\bigr)}.
\]
\end{theorem}

\begin{proof}
Using formula \ref{123},
\begin{align*}
\sum_{a=0}^{\infty} \frac{\phi_\ell(p^a)}{p^{a s}}
&= \sum_{a=0}^{\ell-1} \frac{p^{a}}{p^{a s}} 
   + \sum_{a=\ell}^{\infty} \frac{p^{a}\left(1 - \frac{1}{p}\right)}{p^{a s}} \\
&= \sum_{a=0}^{\ell-1} p^{a(1-s)} 
   + \left(1 - \frac{1}{p}\right) \sum_{a=\ell}^{\infty} p^{a(1-s)} .
\end{align*}
Both sums are geometric series. The first sum equals
\[
\frac{1 - p^{\ell(1-s)}}{1 - p^{1-s}}.
\]
 The second sum equals, provided $|p^{1-s}|<1$ (which holds for $\Re(s)>2$),
\[
\sum_{a=\ell}^{\infty} p^{a(1-s)} = p^{\ell(1-s)} \sum_{k=0}^{\infty} p^{k(1-s)}
= \frac{p^{\ell(1-s)}}{1 - p^{1-s}} .
\]
Hence, 
\begin{equation*}
  \sum_{a=0}^{\infty} \frac{\phi_\ell(p^a)}{p^{a s}}=  \frac{1 - p^{\ell(1-s)-1}}{1 - p^{1-s}}.
\end{equation*}
Because $\phi_\ell$ is multiplicative and absolutely convergent in $\Re(s) > 2 $, its Dirichlet series factors into an Euler product:
\[
\sum_{n=1}^{\infty} \frac{\phi_\ell(n)}{n^{s}}
= \prod_{p \in \mathbb{P}} \left( \sum_{a=0}^{\infty} \frac{\phi_\ell(p^a)}{p^{a s}} \right),
\qquad \Re(s) > 2 .
\]
Each Euler factor equals
\[
\frac{1 - p^{-[\ell(s-1)+1]}}{1 - p^{-(s-1)}} .
\]
Therefore,
\[
\sum_{n=1}^{\infty} \frac{\phi_\ell(n)}{n^{s}}
= \prod_{p} \frac{1 - p^{-[\ell(s-1)+1]}}{1 - p^{-(s-1)}} .
\]
Recall the Euler product for the Riemann zeta function (valid for $\Re(z)>1$)
\[
\zeta(z) = \prod_{p} \frac{1}{1 - p^{-z}} .
\]
For $\Re(s) > 2$, we have
\[
\prod_{p} \bigl(1 - p^{-(s-1)}\bigr) = \frac{1}{\zeta(s-1)}, \qquad
\prod_{p} \bigl(1 - p^{-[\ell(s-1)+1]}\bigr) = \frac{1}{\zeta\!\bigl(\ell(s-1)+1\bigr)} .
\]
Consequently,
\[
\sum_{n=1}^{\infty} \frac{\phi_\ell(n)}{n^{s}}
= \frac{\displaystyle \prod_{p} \bigl(1 - p^{-[\ell(s-1)+1]}\bigr)}
        {\displaystyle \prod_{p} \bigl(1 - p^{-(s-1)}\bigr)}
= \frac{\zeta(s-1)}{\zeta\!\bigl(\ell(s-1)+1\bigr)} .
\]
\end{proof}
Setting $\ell = 1$ in Theorem \ref{thm:euler_product},  we obtain the well‑known Dirichlet series of Euler’s  totient function, see \citep[p.~229]{apostol1976}.

\begin{corollary}[Dirichlet series of $\phi$] \label{cor:classical}
For $\ell = 1$,
\[
\sum_{n=1}^{\infty} \frac{\phi(n)}{n^{s}} = \frac{\zeta(s-1)}{\zeta(s)}, \qquad \Re(s) > 2 .
\]
\end{corollary}

\begin{corollary}[Meromorphic continuation] \label{cor:continuation}
The function
\[
F_\ell(s) := \sum_{n=1}^{\infty} \frac{\phi_\ell(n)}{n^{s}}
\]
admits a meromorphic continuation to the whole complex plane. Its poles are located at:
\begin{enumerate}
    \item $s = 2$ (simple pole with residue $1/\zeta(\ell+1)$);
    \item $s = 1 + \frac{\rho-1}{\ell}$ for every non‑trivial zero $\rho$ of $\zeta$ , these poles are simple.
    \item $s = 1 - \frac{2n+1}{\ell}$ for $n \in \mathbb{N}$, these poles are simple.
\end{enumerate}
\end{corollary}

\begin{proof}
Since $\zeta(z)$ is meromorphic on $\mathbb{C}$ with a simple pole at $z=1$ and zeros at the non‑trivial zeros $\rho$ (and trivial zeros at $z=-2n$), the quotient
\[
F_\ell(s) = \frac{\zeta(s-1)}{\zeta\!\bigl(\ell(s-1)+1\bigr)}
\]
is meromorphic everywhere.
\begin{enumerate}
    \item The numerator $\zeta(s-1)$ has a simple pole at $s-1 = 1$, i.e., at $s=2$. The denominator $\zeta(\ell(s-1)+1)$ is regular at $s=2$ (since $\ell(s-1)+1 = \ell+1 > 1$). Hence $s=2$ is a simple pole of $F_\ell(s)$. The residue is
    \[
    \text{Res}_{s=2} F_\ell(s) = \frac{1}{\zeta(\ell+1)} .
    \]
    \item The denominator vanishes when $\ell(s-1)+1 = \rho$, i.e., $s = 1 + \frac{\rho-1}{\ell}$, for any zero $\rho$ of $\zeta$. If $\rho$ is a non‑trivial zero, these points are poles of $F_\ell(s)$ unless they are cancelled by a zero of $\zeta(s-1)$. Such cancellation would require $s-1 = \rho'$ for some zero $\rho'$ of $\zeta$, i.e., $1 + \frac{\rho-1}{\ell} - 1 = \rho'$, or 
\begin{equation} \label{rho}
        \rho = 1 + \ell \rho'.
\end{equation}
     Consider two subcases.
\begin{itemize}
  \item $\rho'$ non‑trivial: then $\Re(\rho')>0$, so $\Re(\rho)=1+\ell\Re(\rho')>1$, contradicting the classical zero‑free region $\zeta(s)\neq0$ for $\Re(s)\ge1$ (apart from the pole at $s=1$).
  \item $\rho'$ trivial: $\rho'=-2m$ ($m\ge1$). Then $\rho = 1-2m\ell$ is real and $\le -1$. Such a number is not a non‑trivial zero (its real part is not in $(0,1)$) and is never a trivial zero because $1-2m\ell$ is odd when $\ell$ is odd and even but not of the form $-2m'$ when $\ell$ is even; in all cases it fails to be an even negative integer.
\end{itemize}
Thus \ref{rho} cannot hold for any zero $\rho$. Consequently every zero of $\zeta$ gives a genuine simple pole of $F_\ell(s)$ at $s=s_\rho$.
    
    \item For the trivial zeros $\rho = -2n$ ($n \ge 1$), we obtain poles at $s = 1 - \frac{2n+1}{\ell}$. A pole at $s_\rho$ is cancelled (i.e. becomes removable or a pole of lower order) if and only if the numerator also vanishes at that point:
\[
\zeta(s_\rho-1) = \zeta\!\left(\frac{-2n-1}{\ell}\right) = 0.
\]
Since $\frac{-2n-1}{\ell}$ is real, cancellation would require $\frac{-2n-1}{\ell} = -2m$ for some integer $m \ge 1$, i.e. $2n+1 = 2m\ell.$
 This is impossible for any integers $n,m,\ell \ge 1$.  Hence  $\zeta\!\left(\frac{-2n-1}{\ell}\right) \neq 0$.  
Thus the numerator is non‑zero while the denominator vanishes (because $\rho = -2n$ is a zero of $\zeta$). Consequently $F_\ell(s)$ has a pole at $s = s_\rho$. Moreover, all trivial zeros of $\zeta$ are simple; since the numerator is analytic and non‑zero at $s_\rho$, the pole of $F_\ell(s)$ is also simple.
\end{enumerate}
\end{proof}

\begin{corollary}[Poles under RH] \label{cor:RH}
Assume the Riemann hypothesis. Then all non‑trivial poles of $F_\ell(s)$ lie on the vertical line
\[
\Re(s) = 1 - \frac{1}{2\ell}.
\]
\end{corollary}

\begin{proof}
Under RH, every non‑trivial zero $\rho$ of $\zeta$ satisfies $\Re(\rho) = \frac12$. Hence for such a zero,
\[
\Re\!\left(1 + \frac{\rho-1}{\ell}\right) = 1 + \frac{\frac12 - 1}{\ell} = 1 - \frac{1}{2\ell}.
\]
\end{proof}

\begin{corollary}[M\"obius inversion formula] \label{cor:mobius}
For every $n \in \mathbb{N}$,
\[
\phi_\ell(n) = n \sum_{d^{\,\ell} \mid n} \frac{\mu(d)}{d}.
\]
\end{corollary}

\begin{proof}
Consider the Dirichlet series identity from Theorem \ref{thm:euler_product},
\[
\sum_{n=1}^{\infty} \frac{\phi_\ell(n)}{n^{s}} = \frac{\zeta(s-1)}{\zeta\!\bigl(\ell(s-1)+1\bigr)}.
\]
Since $1/\zeta(z) = \sum_{d=1}^{\infty} \mu(d) d^{-z}$, we have
\[
\frac{1}{\zeta\!\bigl(\ell(s-1)+1\bigr)} = \sum_{d=1}^{\infty} \frac{\mu(d)}{d^{\,\ell(s-1)+1}}.
\]
Therefore,
\[
\sum_{n=1}^{\infty} \frac{\phi_\ell(n)}{n^{s}} 
= \zeta(s-1) \sum_{d=1}^{\infty} \frac{\mu(d)}{d^{\,\ell(s-1)+1}}
= \sum_{m=1}^{\infty} \frac{1}{m^{s-1}} \sum_{d=1}^{\infty} \frac{\mu(d)}{d^{\,\ell(s-1)+1}}.
\]
Set $n = m d^{\ell}$. Then $n^{-s} = m^{-(s-1)} d^{-\ell(s-1)-1}$, and the coefficient of $n^{-s}$ is
\[
\sum_{\substack{m,d \ge 1 \\ m d^{\ell} = n}} \mu(d) = \sum_{d^{\ell} \mid n} \mu(d).
\]
Since the Dirichlet series coefficients are unique, see \citep[p.~227]{apostol1976}, multiplying by $n$ gives the desired formula.
\end{proof}

\section{Summatory function and the Riemann hypothesis }
\begin{theorem}\label{thm:main}
Fix an integer $\ell\ge1$. Then: For all $x\ge1$,
\begin{equation} \label{sl}
\Phi_\ell(x)=\frac{x^2}{2\zeta(\ell+1)} - \frac{x^2}{2}\,T(x^{1/\ell}) + O_{\ell}(x\log x),    
\end{equation}
where , for $y\ge1$, $T(y):=\sum_{d>y}\frac{\mu(d)}{d^{\ell+1}}.$
\end{theorem}
\begin{proof}
Start from the definition and rearrange so that
\[
\Phi_\ell(x)=\sum_{n\le x} n\sum_{d^\ell\mid n}\frac{\mu(d)}{d}.
\]
 Interchange sums to get
\[
\Phi_\ell(x)=\sum_{d\le x^{1/\ell}}\frac{\mu(d)}{d}\sum_{\substack{n\le x\\d^\ell\mid n}} n.
\]
For fixed $d$, set $n = d^\ell m$. Then $m$ runs from $1$ to $\left\lfloor x/d^\ell \right\rfloor =: N$. Thus
\[
\sum_{\substack{n \leq x \\ d^\ell \mid n}} n = \sum_{m=1}^N d^\ell m = d^\ell \frac{N(N+1)}{2} = \frac{d^\ell N^2}{2} + \frac{d^\ell N}{2}.
\]

Write $N = x/d^\ell - \theta$ with $0 \leq \theta < 1$. Then
\[
\frac{d^\ell N^2}{2} = \frac{x^2}{2 d^\ell} - x \theta + \frac{d^\ell \theta^2}{2}, \quad \frac{d^\ell N}{2} = \frac{x}{2} - \frac{d^\ell \theta}{2}.
\]

Adding gives
\[
\frac{x^2}{2 d^\ell} + \frac{x}{2} - x \theta - \frac{d^\ell \theta}{2} + \frac{d^\ell \theta^2}{2}.
\]

The terms involving $\theta$ are bounded by $x + d^\ell$ (since $|\theta| \leq 1$). Hence
\[
\sum_{\substack{n \leq x \\ d^\ell \mid n}} n = \frac{x^2}{2 d^\ell} + O(x + d^\ell).
\]

Multiply by $\mu(d)/d$ and sum over $d \leq x^{1/\ell}$:
\[
\Phi_\ell(x) = \sum_{d \leq x^{1/\ell}} \frac{\mu(d)}{d} \left( \frac{x^2}{2 d^\ell} + O(x + d^\ell) \right) = \frac{x^2}{2} \sum_{d \leq x^{1/\ell}} \frac{\mu(d)}{d^{\ell+1}} + \sum_{d \leq x^{1/\ell}} O\left(\frac{x}{d}\right) + \sum_{d \leq x^{1/\ell}} O(d^{\ell - 1}).
\]

For $X = x^{1/\ell}$,
\[
\sum_{d \leq X} \frac{1}{d} = \log X + \gamma + O(1/X) = \frac{1}{\ell} \log x + O(1),
\]
see \citep[p.~6]{Tenenbaum1995}. Therefore the first error term is $O_{\ell}(x \log x)$.

The function $f(t) = t^{\ell - 1}$ is increasing for $\ell \geq 1$, thus the second error term is of order
\[
\sum_{d \leq x^{1/\ell}} d^{\ell - 1} \leq \int_0^{x^{1/\ell}} t^{\ell - 1} dt + (x^{1/\ell})^{\ell - 1} = \frac{x}{\ell} + x^{1 - 1/\ell} = O_{\ell}(x).
\]

Hence,
\[
\Phi_\ell(x) = \frac{x^2}{2} \sum_{d \leq x^{1/\ell}} \frac{\mu(d)}{d^{\ell+1}} + O_{\ell}\left( x \log x \right).
\]

Since $\ell \ge 1$, the Dirichlet series for $1/\zeta$ converges absolutely and 
\[
\frac{1}{\zeta(\ell + 1)} = \sum_{d=1}^\infty \frac{\mu(d)}{d^{\ell + 1}},
\]
see \citep[p.~3]{Titchmarsh1986}.
Thus
\[
\sum_{d \leq x^{1/\ell}} \frac{\mu(d)}{d^{\ell + 1}} = \frac{1}{\zeta(\ell + 1)} - \sum_{d > x^{1/\ell}} \frac{\mu(d)}{d^{\ell + 1}} = \frac{1}{\zeta(\ell + 1)} - T(x^{1/\ell}),
\]
where
\[
T(y) = \sum_{d > y} \frac{\mu(d)}{d^{\ell + 1}}.
\]

Substituting into the expression for $\Phi_\ell(x)$ yields
\[
\Phi_\ell(x) = \frac{x^2}{2 \zeta(\ell + 1)} - \frac{x^2}{2} T(x^{1/\ell}) + O_{\ell}(x \log x).
\]

\end{proof}
\begin{remark}
(1) At $\ell=1$, Equation \ref{sl} reduces to  the classical result
\begin{equation*}
  \sum_{n\le x}\varphi(n)=\frac{x^2}{2\zeta(\ell+1)} + O(x\log x),   
\end{equation*}
see  \textnormal{\citep[p.~39]{Tenenbaum1995}}. This is because $\frac{x^2}{2}\,T(x)=O(x).$ \\
(2) The best known unconditional error term is due to Walfisz     
$$\sum_{n\le x}\varphi(n)= \frac{x^2}{2\zeta(2)}+O( x(\log x)^{\frac{3}{2}} (\log_{2} x)^{\frac{4}{3}}),$$
see  \textnormal{\citep[p.~144]{Walfisz1963}}.
\end{remark}
\begin{theorem}\label{main}
\begin{enumerate}
    \item  If the Riemann hypothesis is true, then  
$$\Phi_\ell(x)= \frac{x^2}{2\zeta(\ell+1)}+O_{\ell,\epsilon}( x).$$
\item For any fixed integer $\ell \geq 1$, if for every $\varepsilon > 0$
\[
\Phi_\ell(x) = \frac{x^2}{2 \zeta(\ell+1)} + O_{\ell, \epsilon }\left(x^{1 - \frac{1}{2\ell} + \varepsilon}\right),
\]
then the Riemann hypothesis is true.
\end{enumerate}
 
\end{theorem}
\begin{proof}
(1)  Fix an integer $\ell \ge1$.  Let RH be true. Then for some $0<\eta<\tfrac12$ there is a constant $C_M(\eta)>0$ such that
\begin{equation} \label{eqM}
 |M(t)|\le C_M(\eta)\,t^{1/2+\eta}\qquad(t\ge1).   
\end{equation}

By Theorem \ref{thm:main},
\[
\;\Phi_\ell(x)=\frac{x^2}{2\zeta(\ell+1)} - \frac{x^2}{2}\,T(x^{1/\ell}) + O_{\ell}(x\log x).
\] 
Using the Stieltjes representation and (M) we have for every $y\ge1$,
\[
\begin{aligned}
|T(y)|
&= \Big|-\frac{M(y)}{y^{\ell+1}} + (\ell+1)\int_y^\infty \frac{M(t)}{t^{\ell+2}}\,dt\Big| \\
&\le \frac{|M(y)|}{y^{\ell+1}} + (\ell+1)\int_y^\infty \frac{|M(t)|}{t^{\ell+2}}\,dt \\
&\le C_M(\eta)\,y^{-\ell-1/2+\eta} \;+\; C_M(\eta)(\ell+1)\int_y^\infty t^{-\ell-3/2+\eta}\,dt.
\end{aligned}
\]
The integral is elementary (since $-\ell-3/2+\eta<-1$), so
\[
\int_y^\infty t^{-\ell-3/2+\eta}\,dt
= \frac{y^{-\ell-1/2+\eta}}{\ell+1/2-\eta}.
\]
Hence
\[
|T(y)| \le C_M(\eta)\, y^{-\ell-1/2+\eta}\,\Big(1 + \frac{\ell+1}{\ell+1/2-\eta}\Big),
\]
so with
\[
A(\ell,\eta):=1 + \frac{\ell+1}{\ell+1/2-\eta}
\]
we obtain the explicit bound
\[
|T(y)| \le C_M(\eta)\,A(\ell, \eta)\, y^{-\ell-1/2+\eta}\quad(y\ge1).
\]  
Put $y=x^{1/\ell}$. Then,
\begin{eqnarray*}
\Big|\frac{x^2}{2}T(x^{1/\ell})\Big|
&\le& \frac{x^2}{2}\cdot C_M(\eta) A(\ell, \eta)\, x^{-(\ell+1/2-\eta)/\ell}\\
&=& \frac{C_M(\eta)A(\ell, \eta)}{2}\, x^{1-1/(2\ell) + \eta/k} \\
&=& D(\ell, \eta) x^{1-1/(2\ell) + \eta/\ell},
\end{eqnarray*}
where
\[
D(\ell, \eta):=\frac{C_M(\eta)A(\ell, \eta)}{2}.
\]
Using 
\[
\sum_{m\le X} m = \frac{X^2}{2} + \frac{X}{2} + O(1),
\]
one gets (after the usual substitution $n=d^\ell m$ and summation over $d\le x^{1/\ell}$)
\[
\Phi_\ell(x)
= \frac{x^2}{2}\sum_{d\le x^{1/\ell}}\frac{\mu(d)}{d^{\ell+1}}
\;+\;\frac{x}{2}\sum_{d\le x^{1/\ell}}\frac{\mu(d)}{d}
\;+\; O\!\Big(\sum_{d\le x^{1/\ell}} d^{\,\ell-1}\Big).
\]
As proved before,
\[
\sum_{d\le x^{1/\ell}} d^{\ell-1} \le \Bigl(\frac{1}{\ell}+1\Bigr)x.
\]
Let \(c_\ell = \frac{1}{\ell}+1\). 
Next we bound the middle term
\[
\frac{x}{2}\sum_{d\le x^{1/\ell}}\frac{\mu(d)}{d}.
\]
Use partial summation: for any $X\ge1$,
\[
\sum_{d\le X}\frac{\mu(d)}{d} = \frac{M(X)}{X} + \int_1^X \frac{M(t)}{t^2}\,dt.
\]
With the bound in Inequality \ref{eqM} we get (for $0<\eta<1/2$)
\[
\Big|\sum_{d\le X}\frac{\mu(d)}{d}\Big|
\le C_M(\eta) X^{-1/2+\eta} + C_M(\eta)\int_1^X t^{-3/2+\eta}\,dt
\le C_M(\eta) X^{-1/2+\eta} + \frac{C_M(\eta)}{1/2-\eta}.
\]
Thus with $X=x^{1/\ell}$,
\[
\frac{x}{2}\Big|\sum_{d\le x^{1/\ell}}\frac{\mu(d)}{d}\Big|
\le \frac{x}{2}\Big( C_M(\eta) x^{-1/(2\ell)+\eta/\ell} + \frac{C_M(\eta)}{1/2-\eta}\Big).
\]
The first term on the right is of the same shape as the tail contribution and can be absorbed into the $D(\ell,\eta)$-term. The second term yields a contribution of size $\frac{C_M(\eta)}{2(1/2-\eta)}\,x.$ 
Collecting the elementary $O(x)$ pieces we obtain the linear coefficient
\[
E(\ell,\eta) \;=\; c_\ell \;+\; \frac{C_M(\eta)}{2(1/2-\eta)}.
\]
Combining the tail bound  and the elementary linear contributions yields precisely the claimed bound.\\
Now, fix \(\varepsilon>0\) and set
\[
\eta := \frac{\ell\varepsilon}{2} \quad(\text{note }0<\eta<1/2\ \text{for small }\varepsilon).
\]
Then the tail-exponent becomes
\[
1-\frac{1}{2\ell} + \frac{\eta}{\ell} \;=\; 1-\frac{1}{2\ell} + \frac{\varepsilon}{2}.
\]
Thus for every $x\ge1$ we have explicitly
\[
\Big|\Phi_\ell(x)-\frac{x^2}{2\zeta(\ell+1)}\Big|
\le D(\ell,\varepsilon)\, x^{1-\tfrac{1}{2\ell}+\tfrac{\varepsilon}{2}} + E(\ell,\varepsilon)\,x.
\]
We have $1-\frac{1}{2\ell}+\frac{\varepsilon}{2}< 1.$ Choose a constant $C(\ell,\varepsilon)$ sufficiently large and a threshold $X_0$ such that $C(\ell,\varepsilon)$   dominates both terms for all $x\ge X_0$. That is
\[
\Big|\Phi_\ell(x)-\frac{x^2}{2\zeta(\ell+1)}\Big| \le C(\ell,\varepsilon)\,x \qquad(x\ge X_0).
\]
(2) Let
\[
M = \frac{1}{2 \zeta(\ell+1)}.
\]

The hypothesis becomes
\[
\Phi_\ell(x) = M x^2 + E_\ell(x), \quad \text{with } E_\ell(x) = O\left(x^{\alpha + \varepsilon}\right) \ \forall \varepsilon > 0,
\]
where
\[
\alpha = 1 - \frac{1}{2\ell}.
\]
For $\Re(s) > \alpha$, consider
\begin{equation}\label{e1}
 G(s) := \int_1^\infty E_\ell(x) x^{-s-1} dx.   
\end{equation}
Since $|E_\ell(x)| \leq C_\varepsilon x^{\alpha + \varepsilon}$ for every $\varepsilon > 0$, the integral converges absolutely for $\Re(s) > \alpha$ (choose $\varepsilon$ such that $\alpha + \varepsilon < \Re(s)$). 
Now, let $K \subset \{ s \in \mathbb{C} : \Re(s) > \alpha \}$ be compact.  
Since the function $s \mapsto \Re(s)$ is continuous and $K$ lies in the open half-plane, there exists $\delta > 0$ such that
\[
\Re(s) \ge \alpha + \delta
\quad \text{for all } s \in K.
\]
 Choose $\varepsilon = \delta/2$. Then for every $s \in K$,
\[
\Re(s) \ge \alpha + 2\varepsilon.
\]
Hence,
\[
\alpha + \varepsilon - \Re(s)
\le
\alpha + \varepsilon - (\alpha + 2\varepsilon)
=
-\varepsilon.
\]
Using the bound $|E_\ell(x)| \le C_\varepsilon x^{\alpha+\varepsilon}$, we obtain
\[
|E_\ell(x)x^{-s-1}|
\le
C_\varepsilon x^{\alpha+\varepsilon-\Re(s)-1}
\le
C_\varepsilon x^{-1-\varepsilon}.
\]
Define
\[
g(x) := C_\varepsilon x^{-1-\varepsilon}.
\]
Since $\varepsilon>0$,
\[
\int_1^\infty g(x)\,dx
=
C_\varepsilon \int_1^\infty x^{-1-\varepsilon}dx
=
\frac{C_\varepsilon}{\varepsilon}
< \infty.
\]
For any $A>1$ and all $s \in K$,
\[
sup_{s \in K}\left|
\int_A^\infty E_\ell(x)x^{-s-1}\,dx
\right|
\le
\int_A^\infty g(x)\,dx.
\]
Since the right-hand side tends to $0$ as $A \to \infty$, it follows that
\[
\int_1^\infty E_\ell(x)x^{-s-1}\,dx
\]
converges uniformly on $K$. Hence, $G(s)$ is analytic in the half-plane $\Re(s) > \alpha$. We aim  to relate $G(s)$ to $F_\ell(s)$. To this end, we show that 
\begin{equation} \label{11}
 \int_1^\infty \Phi_\ell(x) x^{-s-1} dx = \frac{F_\ell(s)}{s}.   
\end{equation}
 Since $\Phi_\ell(x) = \sum_{n \leq x} \phi_\ell(n)$, we can write it as
\[
\Phi_\ell(x) = \sum_{n=1}^\infty \phi_\ell(n) \mathbf{1}_{\{n \leq x\}},
\]
where $\mathbf{1}_{\{n \leq x\}}$ equals 1 if $n \leq x$ and 0 otherwise. Substituting this into the integral gives
\begin{equation}\label{e2}
\int_1^\infty \Phi_\ell(x) x^{-s-1} dx = \int_1^\infty \left( \sum_{n=1}^\infty \phi_\ell(n) \mathbf{1}_{\{n \leq x\}} \right) x^{-s-1} dx.
\end{equation}
 For $\Re(s) > 2$, the series converges absolutely. To see this, note that
\[
|\phi_\ell(n) \mathbf{1}_{\{n \leq x\}} x^{-s-1}| \leq |\phi_\ell(n)| |x^{-\Re(s)-1}|.
\]
Since $\phi_\ell(n) = O(n)$ (in fact, $\phi_\ell(n) \leq n$), the sum $\sum_{n=1}^\infty |\phi_\ell(n)| n^{-\Re(s)}$ converges for $\Re(s) > 2$. Moreover, for each fixed $n$,
\[
\int_1^\infty \mathbf{1}_{\{n \leq x\}} x^{-s-1} dx = \int_n^\infty x^{-s-1} dx
\]
converges absolutely for $\Re(s) > 0$. Hence, by Fubini's theorem, we may interchange the sum and integral in Equation \ref{e2} to obtain
\begin{equation} \label{e3}
 \int_1^\infty \Phi_\ell(x) x^{-s-1} dx = \sum_{n=1}^\infty \phi_\ell(n) \int_1^\infty \mathbf{1}_{\{n \leq x\}} x^{-s-1} dx.   
\end{equation}
For each $n$,
\[
\int_1^\infty \mathbf{1}_{\{n \leq x\}} x^{-s-1} dx = \int_n^\infty x^{-s-1} dx.
\]
Since $\Re(s) > 0$, then
\[
\int_n^\infty x^{-s-1} dx = \frac{n^{-s}}{s}.
\]
By Equation \ref{e3}, it follows that
\[
\int_1^\infty \Phi_\ell(x) x^{-s-1} dx = \sum_{n=1}^\infty \phi_\ell(n) \frac{n^{-s}}{s} = \frac{1}{s} \sum_{n=1}^\infty \phi_\ell(n) n^{-s} = \frac{F_\ell(s)}{s},
\]
This proves Equation \ref{11}.
The series $F_\ell(s) = \sum_{n=1}^\infty \phi_\ell(n) n^{-s}$ converges absolutely for $\Re(s) > 2$ because $|\phi_\ell(n)| \leq n$, so termwise integration is justified.
In addition,
\[
\int_1^\infty M x^2 x^{-s-1} dx = M \int_1^\infty x^{1 - s} dx = \frac{M}{s - 2}, \quad \Re(s) > 2.
\]
Therefore, for $\Re(s) > 2$,
\[
G(s) = \int_1^\infty E_\ell(x) x^{-s-1} dx = \int_1^\infty \left(\Phi_\ell(x) - M x^2 \right) x^{-s-1} dx = \frac{F_\ell(s)}{s} - \frac{M}{s - 2}.
\]
Second, we aim to find an analytic continuation  of $G(s)$ . Define
\[
H(s) := \frac{F_\ell(s)}{s} - \frac{M}{s - 2}.
\]
The right-hand side of the above is meromorphic on $\mathbb{C}$. Since the left-hand side $G(s)$ is analytic for $\Re(s) > \alpha$, equality holds by analytic continuation for all $s$ with $\Re(s) > \alpha$. Hence, $H(s)$ is analytic for $\Re(s) > \alpha$. Now examine the poles of $H(s)$:
\begin{itemize}
    \item The term $\frac{M}{s-2}$ has a simple pole at $s=2$ with residue $M$.
    \item The term $\frac{F_\ell(s)}{s}$ has poles where $F_\ell(s)$ has poles, and possibly at $s=0$ (from the denominator $s$), but $s=0$ is not in the region $\Re(s) > \alpha$ for $k \geq 1$.
\end{itemize}
At $s=2$, $F_\ell(s)$ has a simple pole with residue
\[
\frac{1}{\zeta(\ell+1)} = 2M.
\]
Thus,
\[
\mathrm{Res}_{s=2} \frac{F_\ell(s)}{s}  = M.
\]
Therefore, the residues of the two terms at $s=2$ cancel, so $H(s)$ is analytic at $s=2$.

Consequently, $H(s)$ has no poles in the half-plane $\Re(s) > \alpha$. Since
\[
H(s) = \frac{F_\ell(s)}{s} - \frac{M}{s-2}
\]
and $\frac{M}{s-2}$ is analytic for $s \neq 2$, it follows that $\frac{F_\ell(s)}{s}$ has no poles in $\Re(s) > \alpha$ except possibly at $s=2$, and even there the pole is canceled. Because $1/s$ is analytic and nonzero for $s \neq 0$, we conclude that $F_\ell(s)$ has no poles with $\Re(s) > \alpha$ except at $s=2$.\\

From Corollary 2.5, every nontrivial zero $\rho$ of $\zeta(s)$ gives a pole of $F_\ell(s)$ at
\[
s_\rho = 1 + \frac{\rho - 1}{\ell}.
\]
Thus,
\[
\Re(s_\rho) = 1 + \frac{\beta - 1}{\ell}.
\]
If $\beta > \frac{1}{2}$, then
\[
\Re(s_\rho) > 1 + \frac{\frac{1}{2} - 1}{\ell} = 1 - \frac{1}{2\ell} = \alpha.
\]
Thus, such a pole would lie in the half-plane $\Re(s) > \alpha$, contradicting the fact that $F_\ell(s)$ has no poles there (except at $s=2$). Therefore, we must have
\[
\beta \leq \frac{1}{2}
\]
for every nontrivial zero $\rho$.

The functional equation of the Riemann zeta function implies that if $\rho$ is a nontrivial zero, then so is $1 - \rho$. If $\beta \leq \frac{1}{2}$, then
\[
1 - \beta \geq \frac{1}{2},
\]
so the zeros are symmetric with respect to the critical line $\Re(s) = \frac{1}{2}$. Combining both inequalities yields
\[
\beta = \frac{1}{2}
\]
for every nontrivial zero.

Thus, all nontrivial zeros of $\zeta(s)$ lie on the critical line $\Re(s) = \frac{1}{2}$, which is the Riemann hypothesis.
\end{proof}
\begin{remark}
The trivial zeros of $\zeta(s)$ at $\rho = -2n$ ($n \geq 1$) give poles of $F_\ell(s)$ at
\[
s = 1 + \frac{-2n - 1}{\ell} = 1 - \frac{2n + 1}{\ell}.
\]
For $n \geq 1$, $1 - \frac{2n + 1}{\ell} > \alpha$ if and only if $2n + 1 < \frac{1}{2},$
which is impossible. Hence, these poles lie in $\Re(s) \leq \alpha$ and do not affect the analyticity of $H(s)$ for $\Re(s) > \alpha$.
\end{remark}

\end{document}